\documentclass[12pt]{amsart}
\usepackage{fullpage}
\usepackage[title]{appendix}
\usepackage{graphicx}
\usepackage{subcaption}
\usepackage{amsthm}
\usepackage{mathrsfs} 
\usepackage{amsfonts}
\usepackage{amssymb}
\usepackage{diagbox}
\usepackage{mathtools}
\usepackage{longtable} 
\usepackage{MnSymbol}
\usepackage{indentfirst}
\usepackage[all]{xy}
\usepackage[colorlinks=true,linkcolor=blue]{hyperref}
\usepackage{ytableau}
\usepackage{tikz-cd}
\usepackage{enumerate}
\theoremstyle{plain}
\newtheorem{theorem}{Theorem}[section]

\newtheorem{conjecture}{Conjecture}[section]

\newtheorem{lemma}[theorem]{Lemma}
\newtheorem{corollary}[theorem]{Corollary}
\newtheorem{proposition}[theorem]{Proposition}

\theoremstyle{remark}

\newtheorem{remark}{Remark}[section]

\makeatletter

\newcommand{\Rmnum}[1]{\expandafter\@slowromancap\romannumeral #1@}
\makeatother

\def\ri{{\rm i}}
\def\rd{{\rm d}}

\def\rrw{\rightarrow}

\def\bR{\mathbb R}
\def\bN{\mathbb N}
\def\bZ{\mathbb Z}
\def\bC{{\mathbb C}}

\def\cP{\mathcal{P}}

\def\cD{{\mathcal D}}

\def\rrw{\rightarrow}
\newmuskip\pFqmuskip

\numberwithin{equation}{section}
\allowdisplaybreaks 

\begin{document}

\title[Residue class biases in partitions]{Residue
class biases in unrestricted partitions, partitions into 
distinct parts, and overpartitions}
\author{Michael J.\ Schlosser and Nian Hong Zhou}

\address{Michael J.\ Schlosser:
Fakult\"{a}t f\"{u}r Mathematik, Universit\"{a}t Wien\\
Oskar-Morgenstern-Platz 1, A-1090, Vienna, Austria}%
\email{michael.schlosser@univie.ac.at}%

\address{Nian Hong Zhou: School of Mathematics and Statistics,
The Center for Applied Mathematics of Guangxi,
Guangxi Normal University, Guilin 541004, Guangxi, PR China}
\email{nianhongzhou@outlook.com; nianhongzhou@gxnu.edu.cn}%

\thanks{The first author was partially supported by
Austrian Science Fund FWF
\href{https://doi.org/10.55776/P32305}{10.55776/P32305}.
The second author was partially supported by
Austrian Science Fund FWF
\href{https://doi.org/10.55776/P32305}{10.55776/P32305}
and the
 National Natural Science Foundation of China
 (No.\ 12301423).}%

\begin{abstract}
We prove specific biases in the number of occurrences of parts
belonging to two different residue classes $a$ and $b$, modulo a
fixed non-negative integer $m$, for the sets of unrestricted partitions,
partitions into distinct parts, and overpartitions. These biases follow
from inequalities for residue-weighted partition functions for the
respective sets of partitions.
We also establish asymptotic formulas for the numbers of partitions of
size $n$ that belong to these sets of partitions and have a symmetric
residue class bias (i.e., for $1\le a<m/2$ and  $b=m-a$),
as $n$ tends to infinity.
\end{abstract}
\maketitle

\section{Introduction}
Let $\bN$ and $\bN_0$ denote the sets of positive and non-negative
integers, respectively.
Throughout this paper, we assume $q$ to be a fixed
complex number satisfying $0<|q|<1$.
For any indeterminant $a$ and complex number $c$,
let the \textit{$q$-shifted factorial} (cf.\ \cite{MR2128719})
be defined by
$$
(a;q)_\infty:=\prod_{j\ge 0}(1-aq^j),\quad\text{and}\quad
(a;q)_c:=\frac{(a;q)_\infty}{(aq^c;q)_\infty}.
$$
Products of $q$-shifted factorials are compactly denoted as
$$
(a_1,\ldots, a_m;q)_c:=\prod_{1\le j\le m}(a_j;q)_c
$$
for $m\in\bN$ and $c\in \bC\cup\{\infty\}$.

A partition $\lambda$ of a positive integer $n$ (cf.\ \cite{MR557013})
is a nonincreasing sequence of positive integers (called the parts),
$\lambda=(\lambda_1, \lambda_2,\ldots,\lambda_l)$, such that
$\lambda_1+\lambda_2+\cdots+\lambda_l=n$. The integer
$l=:\ell(\lambda)$ is called the number of parts of $\lambda$.
If $\lambda$ is a partition of $n$ then we write $|\lambda|=n$.
We also define that $\ell(\varnothing)=|\varnothing| = 0$ for the empty
partition, $\varnothing$, of $0$. Let $\cP(n)$ and $\cD(n)$ denote
the sets of all partitions and all partitions with distinct parts (which we
shall call ``distinct partitions'' for short), of $n$, respectively.
Let $\cP=\bigcup_{n\ge 0}\cP(n)$ and $\cD=\bigcup_{n\ge 0}\cD(n)$
denote the sets of all partitions and distinct partitions, respectively.
Let $a,b, m\in\bN$ with $1\le a<b\le m$. We define
$\ell_{a,m}(\lambda)$ to be the number of parts of $\lambda$ congruent
to $a$ modulo $m$, and define
$\ell_{a,m}(\lambda,\mu):=\ell_{a,m}(\lambda)+\ell_{a,m}(\mu)$
for each pair $(\lambda, \mu)\in \cP\times\cD$.

In this paper, we consider the following weighted partition function,
defined as
\begin{equation}
p_n(x,y)=\sum_{\substack{|\lambda|+|\mu|=
n\\(\lambda, \mu)\in \cP\times\cD}}x^{\ell(\lambda)}y^{\ell(\mu)}.
\end{equation}
where $x$ and $y$ are non-negative real numbers, and $n\in\bN_0$.
It is clear that
$p(n)=p_n(1,0)$, $q(n)=p_n(0,1)$ and $\overline{p}(n)=p_n(1,1)$
are the unrestricted (or ordinary) partition function, distinct partition
function, and overpartition function, respectively.
(Overpartitions were first explicitly considered by Brenti
in \cite[Section~3]{MR1212626} who called them ``dotted partitions''.
A thorough independent study of these objects, involving bijections,
generating functions and connections to Bailey chains, was initiated
by Corteel and Lovejoy~\cite{MR2034322}
who gave them the name ``overpartitions'', which we follow.)
We will focus on the residue-weighted biases partition functions
which we define by
\begin{equation}\label{eqm}
p_{n}(a,b,m; x,y):=\sum_{\substack{|\lambda|+|\mu|
=n\\(\lambda, \mu)\in\cP\times \cD\\
\ell_{a,m}(\lambda,\mu)> \ell_{b,m}(\lambda,\mu)}}
x^{\ell(\lambda)}y^{\ell(\mu)}.
\end{equation}
From standard arguments of partition theory (cf.\ \cite{MR557013}),
we have the following generating function for the statistics $\ell(\lambda)$
on $\cP$, $\ell(\mu)$ on $\cD$,
$\ell_{a,m}(\lambda,\mu), \ell_{b,m}(\lambda,\mu)$ and
$|\lambda|+|\mu|$ on $\cP\times \cD$:
\begin{equation}\label{eqm10}
\sum_{(\lambda, \mu)\in\cP\times \cD}
u^{\ell_{a,m}(\lambda,\mu)}v^{\ell_{b,m}(\lambda,\mu)}
x^{\ell(\lambda)}y^{\ell(\mu)}q^{|\lambda|+|\mu|}=
\frac{(-yq;q)_\infty( -uyq^a,-vyq^b,xq^a,xq^b;q^m)_\infty}
{(xq;q)_\infty(-yq^a,-yq^b,uxq^a,vxq^b;q^m)_\infty}.
\end{equation}

Special cases of \eqref{eqm10} are of interest.
The case $(x,y)=(1,0)$, that is
\begin{equation*}
\sum_{\lambda\in\cP}u^{\ell_{a,m}(\lambda)}
v^{\ell_{b,m}(\lambda)}q^{|\lambda|}=
\frac{(q^a,q^b;q^m)_\infty}{(q;q)_\infty(uq^a,vq^b;q^m)_\infty},
\end{equation*}
was studied by Chern who in
\cite[Theorem~1.3]{MR4372179} proved that
$$p_{n}(a,b,m; 1,0)\ge p_{n}(b,a,m; 1,0)$$
holds for all $n\ge 0$ and $1\le a<b\le m$. Already earlier,
Kim--Kim--Lovejoy \cite{MR4107768} proved a phenomenon of parity bias
for integer partitions, namely: the number of
partitions of $n$ ($n>0$, $n\neq 2$) with more odd parts than even parts
is greater than the number of partitions of $n$ ($n>0$, $n\neq 2$) with
more even parts than odd parts. In our notation, they proved
\begin{equation}\label{eqkkl}
p_{n}(1,2,2; 1,0)\ge p_{n}(2,1,2; 1,0)
\end{equation}
for all $n\ge 0$, where the inequality is strict except for $n\in \{0, 2\}$.
Similar phenomena were shown by Kim--Kim \cite{MR4308132} and include
the following inequalities:
$$p_{n}(1,m,m; 1,0)\ge p_{n}(m,1,m; 1,0)\quad\text{and}\quad
p_{n}(1,m-1,m; 1,0)\ge p_{n}(m-1,1,m; 1,0),$$
for all integers $m\ge 2$ and $n\ge 0$.

While the inequality \eqref{eqkkl} of partitions holds for all $n\ge 0$,
the same inequality holds only from a certain integer on
when restricting to the set of distinct partitions.
In fact, Kim--Kim--Lovejoy \cite{MR4107768} conjectured
(in our notation) that
$$p_{n}(1,2,2; 0,1)> p_{n}(2,1,2; 0,1)$$
for all $n>19$ (only); this has been proved by
Banerjee--Bhattacharjee--Dastidar--Mahanta--Saikia  \cite{MR4396558}
by using a combinatorial approach.

In this paper, we are interested in the biases of $p_{n}(a,b,m; x,y)$
in terms of different residue classes $a$ and $b$ of a fixed modulus $m$.
Our main results are the following three Theorems~\ref{mth1}--\ref{th2}.
\begin{theorem}\label{mth1}Let $a,b,m$ be any integers such that
$1\le a<b\le m$. For any $x\ge 1$, $y\ge 0$ and $n\ge 0$, we have
$$p_{n}(a,b,m; x,y)\ge p_{n}(b,a,m; x,y).$$
\end{theorem}
We point out that the case $(x,y)=(1,1)$ of Theorem~\ref{mth1},
that is the case related to the generating function
\begin{equation*}
\sum_{(\lambda, \mu)\in\cP\times \cD}
u^{\ell_{a,m}(\lambda,\mu)}v^{\ell_{b,m}(\lambda,\mu)}
q^{|\lambda|+|\mu|}=
\frac{(-q;q)_\infty(-uq^a,-vq^b,q^a,q^b;q^m)_\infty}
{(q;q)_\infty(-q^a,-q^b,uq^a,vq^b;q^m)_\infty},
\end{equation*}
is the overpartitions analogue of the aforementioned result by
Chern \cite{MR4372179}. We also establish some new results for the bias
of distinct partitions, belonging to the case $(x,y)=(0,1)$.
In particular, for any $x\ge 0$ and $y=1$, we have the following theorem.
\begin{theorem}\label{mth2}Let $x\ge 0$ and $a,b,m$ be integers
such that $1\le a<b\le m$. Assume that there exists a positive integer $k$
with $k\mid (b-a)$ such that neither of the congruences
$$2^{h}k\equiv a\pmod m\;\;\text{and}\;\; 2^{h}k\equiv b\pmod m,$$
possesses a solution for any $h\in\bN$. Then for all integers $n\ge 0$ we have
$$p_{n}(a,b,m; x,1)\ge p_{n}(b,a,m; x,1).$$
\end{theorem}
\begin{remark}Let $a,b,m$ be integers such that $1\le a<b\le m$.
  It is not difficult to check that the following triples $(a,b,m)$ meet
  the conditions in Theorem~\ref{mth2}.
\begin{enumerate}[(1)]
\item Any even integer $m$ and odd integers $a, b$. (In this case, $k=2$
  guarantees the non-existence of solutions for any of the two congruences.)
\item Any integers $m,a,b$ such that $\gcd(a,b, m)$
  has an odd factor $\ge 3$. (In this case, one can take $k=1$.)
\end{enumerate}
\end{remark}

We finally prove the following asymptotic formulas for weighted
partitions with biases of \emph{symmetric} residue classes (by which we
mean that  in the biases of $p_n(a,b,m;x,y)$ the two residue classes
$a$ and $b$ satisfy the symmetry $b=m-a$ where $1\le a<m/2$).
Below, when writing $a/2m$ we mean $a/(2m)$, and other fractions
in one line notation (with denominators consisting of products) are
to be similarly interpreted in the remainder of this paper. 
\begin{theorem}\label{th2}For any given $1\le a< m/2$ and
  $(x,y)\in\{(1,0),(0,1),(1,1)\}$ we have
$$p_{n}(a,m-a,m; x,y)\sim c_{a,m}(x)p_n(x,y),$$
as $n\rrw +\infty$, where
$$\left(c_{a,m}(0), c_{a,m}(1)\right)=
\left(\frac{1}{2},~\frac{\psi\left(1/2+a/2m\right)-
    \psi\left({a}/{2m}\right)}{2\pi\csc(a\pi/m)}\right),$$
and where $\psi(x)=\Gamma'(x)/\Gamma(x)$ is the digamma function.
\end{theorem}
The proof of Theorem~\ref{th2} employs a residue class analogue of
Ingham's Tauberian theorem \cite[Theorem~1]{MR5522}, which, to the best
of the authors' knowledge, has not appeared before in the literature.
We state it as the following theorem and believe that it is of independent
interest. We give its proof in Section~\ref{sec31}.
\begin{theorem}\label{propit}Let $f(q)=\sum_{n\ge 0}c_nq^n$ be a
power series whose coefficients $c_n$ are non-negative. Suppose that
for some positive integer $m$ and any integer $1\le h<m$, one has
\begin{enumerate}[(i)]
  \item \label{c1} $c_{n+m}\ge c_n$ for all $n\ge 0$, and
  \item \label{c2} there exist constants $\gamma\in\bR$ and
    $\alpha, \beta,\rho\in\bR_+$ such that
    $$f(e^{-z})\sim \alpha z^{\gamma}\exp\left(\beta z^{-\rho}/\rho\right)
    \;\;\text{and}\;\; f(e^{-z}e^{2\pi\ri h/m})=o\left(f(e^{-z})\right),$$
    when $z\rrw 0$ in every fixed Stolz angle
    $|\arg (z)|\le \Delta$ $(0<\Delta<\pi/2)$.
\end{enumerate}
Then, as $n\to \infty$,
$$c_n\sim \frac{\alpha \beta^{\frac{1+2\gamma}{2(1+\rho)}}}
{\sqrt{2\pi(1+\rho)}}n^{-\frac{1+2\gamma}{2(1+\rho)}-\frac{1}{2}}
\exp\left((1+1/\rho)\beta^{1/(1+\rho)} n^{\rho/(1+\rho)}\right).$$
\end{theorem}

The rest of the paper is organized as follows. In Section~\ref{sec2},
we give the proof of Theorems~\ref{mth1} and \ref{mth2}. This is achieved
with the help of some $q$-series transformations and a new overpartitions
analogue of a theorem of Andrews \cite[Theorem~3]{MR289445}.
In Section~\ref{sec3}, we first provide the proof of our residue class
analogue of Ingham's Tauberian theorem, i.e., of Theorem~\ref{propit}.
Then we turn to partition functions with bias of symmetric residue classes
and establish their generating functions and asymptotics. We close in
Section~\ref{sec4} with some remarks on residue class biases in
distinct partitions.

\section{The proofs of Theorems \ref{mth1} and \ref{mth2}}\label{sec2}

In this section, we provide proofs of Theorems \ref{mth1} and \ref{mth2}.
To study $p_{n}(a,b,m; x,y)$, we first derive its generating function.
Using Cauchy's $q$-binomial theorem
(cf.\ \cite[Appendix~(II.3)]{MR2128719}):
$$\frac{(\alpha t;q)_\infty}{(t;q)_\infty}
=\sum_{n\ge 0}\frac{(\alpha;q)_nt^n}{(q;q)_n},$$
valid for $|t|<1$, we have
$$\frac{(-uyq^a,-vyq^b;q^m)_\infty}{(uxq^a,vxq^b;q^m)_\infty}
=\sum_{n_1\ge 0}\frac{(-y/x;q^m)_{n_1}(xuq^{a})^{n_1}}{(q^m;q^m)_{n_1}}
\sum_{n\ge 0}\frac{(-y/x;q^m)_{n}(xvq^{b})^{n}}{(q^m;q^m)_{n}}.$$
Therefore, by the definition \eqref{eqm} for $p_{n}(a,b,m; x,y)$
and \eqref{eqm10}, we obtain
\begin{align}\label{eqmgp}
\sum_{n\ge 0}p_{n}(a,b,m; x,y)q^n
  &=\frac{(-yq;q)_\infty(xq^a,xq^b;q^m)_\infty}
    {(xq;q)_\infty(-yq^a,-yq^b;q^m)_\infty}\nonumber\\
  &\quad\times\sum_{\substack{n_1,n\ge 0\\ n_1> n}}
  \frac{x^{n_1}(-y/x;q^m)_{n_1}x^{n}(-y/x;q^m)_{n}q^{an_1+bn}}
  {(q^m;q^m)_{n_1}(q^m;q^m)_{n}}.
\end{align}

Writing $n_1=k+1+n$ in \eqref{eqmgp}, and using the simple
identity $(\alpha; q)_{n+1+k}=(1-\alpha)(\alpha q;q)_{n+k}$, we get
\begin{align}\label{eqmss0}
(1-q^m)\sum_{n\ge 0}p_{n}(a,b,m; x,y)q^n
  &=\frac{(-yq;q)_\infty(xq^a,xq^b;q^m)_\infty}
    {(xq;q)_\infty(-yq^a,-yq^b;q^m)_\infty}(x+y)q^a\nonumber\\
  &\quad\times\sum_{n,k\ge 0}
    \frac{(-q^my/x;q^m)_{n+k}(-y/x;q^m)_n x^{2n+k}q^{(a+b)n+ka}}
    {(q^{2m};q^m)_{n+k}(q^m;q^m)_{n}}.
\end{align}
Equation \eqref{eqmss0} immediately yields the following monotonicity result,
which we will use for deriving asymptotic formulas for $p_{n}(a,m-a,m; x,y)$
in Section~\ref{sec3}.
\begin{lemma}\label{promt}For any $x,y\ge 0$ and $n\ge 0$, we have
\begin{equation*}
p_{n+m}(a,b,m; x,y)\ge p_{n}(a,b,m; x,y).
\end{equation*}
\end{lemma}

Equation \eqref{eqmss0} also implies that
\begin{align*}
  &\sum_{n\ge 0}
  \left(p_{n}(a,b,m; x,y)-p_{n}(b,a,m; x,y)\right)q^n\nonumber\\
  &=\frac{1}{1-q^m}(1-q^{b-a})\frac{(-yq;q)_\infty(xq^a,xq^b;q^m)_\infty}
  {(xq;q)_\infty(-yq^a,-yq^b;q^m)_\infty}(x+y)q^a\nonumber\\
  &\quad\times\sum_{n,k\ge 0}\frac{(1-q^{(k+1)(b-a)})(-q^my/x;q^m)_{n+k}
    (-y/x;q^m)_n x^{2n+k}q^{(a+b)n+ak}}{(1-q^{b-a})(q^{2m};q^m)_{n+k}(q^m;q^m)_{n}}.
\end{align*}
Note that the $q$-series expansion of the above infinite double sum
has non-negative coefficients, with constant term equal to $1$,
provided that $1\le a<b\le m$ and $x,y\ge 0$, because of
$$\frac{1-q^{(k+1)(b-a)}}{1-q^{b-a}}=\sum_{0\le j\le k}q^{j(b-a)}.$$
This immediately implies the following lemma.
\begin{lemma}\label{eqmprop}Let $1\le a<b\le m$ and $x,y\ge 0$ with $x+y> 0$.
Define
$$f_{a,b,m,x,y}(q)=(1-q^{b-a})\frac{(-yq;q)_\infty}{(xq;q)_\infty}
\frac{(xq^a,xq^b;q^m)_\infty}{(-yq^a,-yq^b;q^m)_\infty}.$$
Then
$$p_{n}(a,b,m; x,y)\ge p_{n}(b,a,m; x,y)\ge 0$$
for all integers $n\ge 0$, provided that all the Taylor coefficients of
$f_{a,b,m,x,y}(q)$ about $q=0$ are non-negative.
\end{lemma}
Lemma~\ref{eqmprop} will be used to give the proof of Theorem~\ref{mth1}
in the cases $(a,b)\neq (1,2)$, and the proof of Theorem~\ref{mth2}.
To achieve that we will need the following proposition.
\begin{proposition}\label{pron12}
For $x\ge 1, y\ge 0$ and $(a,b)\neq (1,2)$, all the Taylor coefficients of
$f_{a,b,m,x,y}(q)$ about $q$ are non-negative.
\end{proposition}
\begin{proof}
Note that for $1<a<b\le m$ we have
$$f_{a,b,m,x,y}(q)=\frac{1-q^{b-a}}{1-xq}\frac{(-yq;q^m)_\infty}
{(xq^{m+1};q^m)_\infty}\prod_{\substack{1\le j\le m\\ j\not\in\{1,a,b\}}}
\frac{(-yq^j;q^m)_\infty}{(xq^j;q^m)_{\infty}}$$
and
$$\frac{1-q^{b-a}}{1-x q}=
\left(1+\sum_{j\ge 1}(x^{j}-x^{j-1})q^j\right)
\sum_{0\le h<b-a}q^{h};$$
and for $a=1$ with $b>2$, we have
$$f_{a,b,m,x,y}(q)=\frac{1-q^{b-1}}{1-x q^{b-1}}
\frac{(-yq^{b-1};q^m)_\infty}{(xq^{m+b-1};q^m)_\infty}
\prod_{\substack{1\le j\le m\\ j\not\in\{1,b-1, b\}}}
\frac{(-yq^j;q^m)_\infty}{(xq^j;q^m)_{\infty}}$$
and
$$\frac{1-q^{b-1}}{1-x q^{b-1}}=1+\sum_{j\ge 1}(x^{j}-x^{j-1})q^{(b-1)j}.$$
Since $x\ge 1$, it follows that all the Taylor coefficients of
$f_{a,b,m,x,y}(q)$ about $q=0$ are non-negative.
This completes the proof of the proposition.
\end{proof}
To prove Theorem~\ref{mth1} for the remaining case $(a,b)=(1,2)$
we use another method. For any $1\le a<b\le m$, by writing
$n_1=k+1+n$ in \eqref{eqmgp} and using the elementary identity
$$(\alpha; q)_{n+1+k}=(\alpha;q)_n(\alpha q^{n};q)_{k+1},$$
we find that
\begin{align*}
\sum_{n\ge 0}p_{n}(a,b,m; x,y)q^n&=
\frac{(-yq;q)_\infty(xq^a,xq^b;q^m)_\infty}
{(xq;q)_\infty(-yq^a,-yq^b;q^m)_\infty}\nonumber\\
&\quad\times\sum_{n\ge 0}\frac{(-y/x;q^m)_{n}^2
(x+yq^{mn})x^{2n}q^{(a+b)n}}{(q^m;q^m)_{n}^2}
\sum_{k\ge 0}\frac{(-q^{m(n+1)}{y}/{x};q^m)_{k}
x^kq^{a(k+1)}}{(q^{m(n+1)};q^m)_{k+1}}.
\end{align*}
Thus
\begin{align}\label{eqmss}
&\sum_{n\ge 0}\left(p_{n}(a,b,m; x,y)-p_{n}(b,a,m; x,y)\right)q^n\nonumber\\
  &=\frac{(-yq;q)_\infty(xq^a,xq^b;q^m)_\infty}
    {(xq;q)_\infty(-yq^a,-yq^b;q^m)_\infty}
    \sum_{n\ge 0}\frac{(-y/x;q^m)_{n}^2 (x+yq^{mn})x^{2n}q^{(a+b)n}}
    {(q^m;q^m)_{n}^2}\nonumber\\
  &\quad\times\left(\sum_{k\ge 0}\frac{(-q^{m(n+1)}{y}/{x};q^m)_{k}
    x^kq^{a(k+1)}}{(q^{m(n+1)};q^m)_{k+1}}-
    \sum_{k\ge 0}\frac{(-q^{m(n+1)}{y}/{x};q^m)_{k}x^kq^{b(k+1)}}
    {(q^{m(n+1)};q^m)_{k+1}}\right).
\end{align}
From Equation~\eqref{eqmss} it is obvious that to prove
$p_{n}(1,2,m; x,y)\ge p_{n}(2,1,m; x,y)$ it suffices to prove the
following theorem, which restricts to such pairs $(a,b)$ where $b$
is a multiple of $a$ (and thus can be replaced by $ab$),
which is more than we need for the case $(a,b)=(1,2)$.
\begin{theorem}\label{maino}
Let $x\ge 1$, $y\ge 0$ and $a,b\in\bN$. For all integers $m,s\ge 1$ the
$q$-series expansion of
$$\sum_{k\ge 0}\frac{(-yq^{s}/x;q^m)_{k}x^{k}q^{a(k+1)}}{(q^{s};q^m)_{k+1}}
-\sum_{k\ge 0}\frac{(-yq^{s}/x;q^m)_{k}x^{k}q^{ab(k+1)}}{(q^{s};q^m)_{k+1}}$$
has non-negative coefficients.
\end{theorem}
We note that the $(x,y,a,b)=(1,0,1,2)$ special case of Theorem~\ref{maino}
was conjectured by Chern~\cite[Conjecture~4.2]{MR4372179} and
subsequently proved by Binner~\cite[Section~2]{Binner}, which we state
as the following corollary:
\begin{corollary}
For all integers $m,s\ge 1$, the $q$-series expansion of
$$\sum_{k\ge 0}\frac{q^k(1-q^k)}{(q^{s};q^m)_k}$$
has non-negative coefficients.
\end{corollary}
\begin{proof}[Proof of Theorem~\ref{maino}]
Using the Heine transformation \cite[Appendix~(III.2)]{MR2128719}
  (valid for $|z|<1$ and $|\gamma/\beta|<1$)
$$\sum_{n\ge 0}\frac{(\alpha,\beta;q)_n}{(\gamma,q;q)_n}z^n=
\frac{(\gamma/\beta,\beta z;q)_\infty}{(\gamma,z;q)_\infty}
\sum_{n\ge 0}\frac{(\alpha\beta z/\gamma,\beta;q)_n}
{(\beta z,q;q)_n}(\gamma/\beta)^n,$$
and making the substitution $\beta\mapsto q, \gamma\mapsto \gamma q$,
one readily obtains
$$\sum_{n\ge 0}\frac{(\alpha;q)_n}{(\gamma;q)_{n+1}}z^n=
\sum_{n\ge 0}\frac{(\alpha z/\gamma;q)_n}{(z;q)_{n+1}}\gamma^n$$
(which can be referred to as Fine's transformation
\cite[Equation~(6.3)]{MR0956465}, as Fine thoroughly studied this type
of series in his monograph on basic hypergeometric series), valid for $|z|<1$
and $|\gamma|<1$.
Thus
$$\sum_{k\ge 0}\frac{(-yq^{s}/x;q^m)_{k}(q^a(xq^{a})^k-
  q^{ab}(xq^{ab})^k)}{(q^{s};q^m)_{k+1}}=
\sum_{k\ge 0}\left(\frac{q^a(-yq^a;q^m)_{k}}{(xq^{a};q^m)_{k+1}}-
  \frac{q^{ab}(-yq^{ab};q^m)_{k}}{(xq^{ab};q^m)_{k+1}}\right)q^{sk}.$$
Note that, for any integer $k \ge 0$,
\begin{equation*}
  \frac{(-yq^a;q^m)_{k}xq^a}{(xq^a;q^m)_{k+1}}-
  \frac{(-yq^{ab};q^m)_{k}xq^{ab}}{(xq^{ab};q^m)_{k+1}}=
  \sum_{h\ge 1}\left(\frac{(-yq^a;q^m)_{k}(xq^a)^{h}}{(xq^a;q^m)_{k}}
  -\frac{(-yq^{ab};q^m)_{k}(xq^{ab})^{h}}{(xq^{ab};q^m)_{k}}\right)q^{(h-1)km}.
\end{equation*}
The proof of Theorem~\ref{maino} now is an immediate consequence of
Theorem~\ref{pro1}, which we provide subsequently.
\end{proof}
The following result is an overpartition analogue of a result by
Andrews~\cite[Theorem~3]{MR289445}.
\begin{theorem}\label{pro1}Let $(a_j)_{j\ge 0},(b_j)_{j\ge 0}$ be two
increasing sequences of positive integers with $b_0\equiv 0 \pmod {a_0}$,
$b_0>a_0$ and $b_j-a_j\equiv 0\pmod {a_0}$ for each $j\ge 1$. For any
$n\in\bN_0$, $h\in\bN_0$ and $x\ge 1, y\ge 0$, the $q$-series expansion of
$$(xq^{a_0})^h\prod_{0\le j\le n}\frac{1+yq^{a_j}}{1-xq^{a_j}}
-(xq^{b_0})^h\prod_{0\le j\le n}\frac{1+yq^{b_j}}{1-xq^{b_j}}$$
has non-negative coefficients.
\end{theorem}
\begin{proof}For $n=0$,  the $q$-series expansion of
\begin{align*}
&\frac{1+yq^{a_0}}{1-xq^{a_0}} (xq^{a_0})^h-
\frac{1+yq^{b_0}}{1-xq^{b_0}} (xq^{b_0})^h\\
&=(xq^{a_0})^h-(xq^{b_0})^h+(1+y/x)
\sum_{k>h}((xq^{a_0})^{k}-(xq^{b_0})^k)\\
&=x^hq^{a_0h}+(1+y/x)\sum_{h<k< b_0h/a_0}x^{k}q^{a_0k}
+yx^{h-1}q^{b_0h}\\
  &\quad+(1+y/x)\sum_{\substack{k\ge h\\ b_0\nmid a_0k}}x^kq^{a_0k}
  +(1+y/x)\sum_{k\ge h}(x^{b_0k/a_0}-x^{k})q^{b_0k}
\end{align*}
has non-negative coefficients (since $x\ge 1$).
Note that
\begin{align}\label{eqfffm1}
&(xq^{a_0})^h\prod_{0\le j\le n}\frac{1+yq^{a_j}}
{1-xq^{a_j}}-(xq^{b_0})^h\prod_{0\le j\le n}
\frac{1+yq^{b_j}}{1-xq^{b_j}}\nonumber\\
&=\left(\frac{1+y q^{a_n}}{1-xq^{a_n}}-
\frac{1+y q^{b_n}}{1-xq^{b_n}}\right) (xq^{a_0})^h
\prod_{0\le j< n}\frac{1+yq^{a_j}}{1-xq^{a_j}}\nonumber\\
  &\quad+\frac{1+y q^{b_n}}{1-xq^{b_n}}
    \left( (xq^{a_0})^h\prod_{0\le j\le n-1}
    \frac{1+yq^{a_j}}{1-xq^{a_j}}
    -(xq^{b_0})^h\prod_{0\le j\le n-1}
    \frac{1+yq^{b_j}}{1-xq^{b_j}}\right),
\end{align}
and
\begin{equation}\label{eqfffm}
  \left(\frac{1+y q^{a_n}}{1-xq^{a_n}}-\frac{1+y q^{b_n}}{1-xq^{b_n}}\right)
  \prod_{0\le j< n}\frac{1+yq^{a_j}}{1-xq^{a_j}}
  =\frac{ 1-q^{b_n-a_n}}{1-xq^{a_0}}\frac{q^{a_n}(x+y)}{1-x q^{b_n}}
  \prod_{0\le j< n}\frac{1+yq^{a_j}}{1-xq^{a_{j+1}}}.
\end{equation}
Now, since the  $q$-series expansion of the factor
$$\frac{1-q^{b_n-a_n}}{1-xq^{a_0}}=
\left(1+\sum_{j\ge 1}(x^{j}-x^{j-1})q^{a_0j}\right)
\sum_{0\le j< (b_n-a_n)/a_0}q^{a_0 j}$$
has non-negative coefficients, we can conclude that the $q$-series expansion
of \eqref{eqfffm} has non-negative coefficients as well.
Using \eqref{eqfffm1}, the proof is established by induction.
\end{proof}

We are now ready to give the proof of Theorem~\ref{mth2}.
\begin{proof}[Proof of Theorem \ref{mth2}]
Since both of the congruence equations
$$2^{h+1}k\equiv a\pmod m\;\;\text{and}\;\; 2^{h+1}k\equiv b\pmod m$$
possess no solution for $h\in\bN_0$, we have that both of the equations
$$2^{h_1}k=m\ell_1+a\;\; \text{and}\;\; 2^{h_2}k=m\ell_2+b$$
possess no non-negative integer solutions $(h_1,\ell_1)$ and $(h_2, \ell_2)$.
This yields
$$\{2^hk: h\in \bN_0\}\subseteq \{j\in\bN: j\not\equiv a, b\pmod m\}.$$
On the other hand, since $b>a$ and $k\mid(b-a)$, it is clear that the
$q$-series expansion of $(1-q^{b-a})/(1-q^{k})$ has non-negative coefficients.
Thus
$$f_{a,b,m,x,1}(q)=\frac{(1-q^{k})(-q;q)_\infty}
{(-q^a,-q^{b};q^m)_\infty}\cdot\frac{1-q^{b-a}}{1-q^k}
\prod_{\substack{1\le j\le m\\ j\not\in\{a,b\}}}\frac{1}{(xq^j;q^m)_{\infty}}$$
has non-negative coefficients in its $q$-expansion too, since $x\ge 0$ and
the $q$-series expansion of 
\begin{equation*}
\frac{(1-q^{k})(-q;q)_\infty}{(-q^a,-q^{b};q^m)_\infty}
=(1-q^{k})\prod_{h\ge 0}(1+q^{2^hk})
\prod_{\substack{j\ge 1\\ j\not\equiv a,b\pmod m\\ j\neq 2^{h}k, h\in\bN_0}}
  \left(1+q^j\right)=
  \prod_{\substack{j\ge 1\\ j\not\equiv a,b\pmod m\\ j\neq 2^{h}k, h\in\bN_0}}
\left(1+q^j\right)
\end{equation*}
has non-negative coefficients. Therefore, by using
Lemma~\ref{eqmprop} and Proposition~\ref{pron12} we get
$$p_{n}(a,b,m;x,1)-p_{n}(b,a,m;x,1)\ge 0,$$
which completes the proof of Theorem~\ref{mth2}.
\end{proof}

\section{Weighted partitions with bias of symmetric residue classes}\label{sec3}
Recall that $p_n(a,b,m;x,y)$ is defined by \eqref{eqm}.
Let $(x,y)\in S$ where $S=\{(0,1), (1,0), (1,1)\}$, and let $a, m$
be positive integers such that $1\le a<m$ and $m\neq 2a$. We write
$$G_{xy}(a,m;q)=\sum_{n\ge 0}p_n(a,m-a,m;x,y)q^n$$
for the generating function of weighted partitions with
bias of \emph{symmetric} residue classes.
The first objective of this section, rather than just using \eqref{eqmgp},
is to derive forms
for the generating functions $G_{xy}(a,m;q)$, where $(x,y)\in S$,
which will be convenient
for establishing asymptotic formulas for $p_n(a,m-a,m;x,y)$,
where $(x,y)\in S$.

\subsection{Generating functions for weighted partitions with bias of symmetric residue classes}
To obtain the generating functions for $G_{xy}(a,m;q)$ with
$(x,y) \in S$, we require the following identities:
\begin{equation}\label{jacobi}
(-\zeta, -q/\zeta, q;q)_\infty=
\sum_{n\in\bZ}q^{\frac{n(n-1)}{2}}\zeta^n,
\end{equation}
\begin{equation}\label{inersejacobi}
\frac{(q;q)_\infty^2}{(\zeta, q/\zeta;q)_\infty}=
\sum_{n\in\bZ} \frac{(-1)^nq^{\frac{n(n+1)}2}}{1-\zeta q^n},
\end{equation}
and
\begin{equation}\label{lemfrphif}
\frac{(-q/\zeta,-\zeta;q)_\infty}{(q/\zeta,\zeta;q)_\infty}=
\frac{(-q;q)_\infty^2}{(q;q)_\infty^2}\left(1+2\sum_{n\ge 1}
\frac{\zeta^n+(q/\zeta)^n}{1+q^n}\right).
\end{equation}
The identity \eqref{jacobi} is the well-known Jacobi triple product
identity (cf.\ \cite[Appendix~(II.28)]{MR2128719}). The identity
\eqref{inersejacobi} is the reciprocal of a theta function, which
is expanded in terms of partial fractions by the Mittag--Leffler theorem.
(For details, see Ramanujan's lost notebook~\cite[Entry~3.2.1]{MR2952081} or
Garvan~\cite[Equation~(7.15)]{MR920146}.) The identity
\eqref{lemfrphif} is a special case of an identity originally
due to Kronecker (cf.\ \cite[pp.~70--71]{Weil}) and can be easily
derived from Ramanujan's $_1\psi_1$ summation formula
(cf.\ \cite[Appendix~(II.29)]{MR2128719}).
\begin{proof}[Proof of identity \eqref{lemfrphif}]
Ramanujan's $_1\psi_1$ summation formula is
$$\sum_{n\ge 0}\frac{(a;q)_n}{(b;q)_n}\zeta^n+
\sum_{n\ge 1}\frac{(q/b;q)_{n}}{(q/a;q)_n}(b/a\zeta)^{n}=
\frac{(b/a,q/a\zeta,a\zeta,q;q)_\infty}{(b,b/a\zeta,q/a,\zeta;q)_\infty},$$
provided $|\zeta|<1$ and $|b/a\zeta|<1$.
Letting $b\mapsto aq$ this reduces to
Kronecker's bilateral summation
$$\sum_{n\ge 0}\frac{1-a}{1-aq^n}\zeta^n+
\sum_{n\ge 1}\frac{1-1/a}{1-q^n/a}(q/\zeta)^{n}=
\frac{(q,q/a\zeta,a\zeta,q;q)_\infty}{(aq,q/\zeta,q/a,\zeta;q)_\infty}.$$
Letting $a\to -1$, we obtain after some rewriting \eqref{lemfrphif}.
\end{proof}

We now give the generating functions for $G_{xy}(a,m;q)$ with $(x,y) \in S$.
\begin{lemma}\label{lemmm1}We have
$$
G_{01}(a,m;q)=\frac{(q^2;q^2)_\infty}
{(-q^a,-q^{m-a},q^m;q^m)_\infty(q;q)_\infty}
\sum_{n\ge 1}q^{\frac{mn(n-1)}{2}+n a},$$
$$
G_{10}(a,m;q)=\frac{(q^a,q^{m-a};q^m)_\infty}
{(q;q)_\infty(q^m;q^m)_\infty^2}\sum_{n\ge 0}(-1)^{n}
\frac{q^{m\frac{n(n+1)}{2}+mn+a}}{1-q^{mn+a}},
$$
and
$$
G_{11}(a,m;q)=2\frac{(q^2;q^2)_\infty(q^{2m};q^{2m})_\infty^2}
{(q;q)_\infty^2(q^m;q^m)_\infty^4}\frac{(q^a,q^{m-a};q^m)_\infty}
{(-q^a,-q^{m-a};q^m)_\infty}\sum_{n\ge 1}\frac{q^{an}}{1+q^{mn}}.
$$
\end{lemma}
\begin{proof}
Recall that
\begin{equation*}
  G_{xy}(a,m;q)=\sum_{\substack{(\lambda, \mu)\in\cP\times \cD\\
      \ell_{a,m}(\lambda,\mu)-\ell_{b,m}(\lambda,\mu)>0}}
  x^{\ell(\lambda)}y^{\ell(\mu)}q^{|\lambda|+|\mu|},
\end{equation*}
and, by the $u=t$, $v=t^{-1}$ case of \eqref{eqm10},
\begin{equation*}
  \sum_{(\lambda, \mu)\in\cP\times \cD}
  t^{\ell_{a,m}(\lambda,\mu)-\ell_{b,m}(\lambda,\mu)}
  x^{\ell(\lambda)}y^{\ell(\mu)}q^{|\lambda|+|\mu|}=
  \frac{(-yq;q)_\infty( xq^a,xq^b;q^m)_\infty ( -tyq^a,-t^{-1}yq^b;q^m)_\infty}
  {(xq;q)_\infty(-yq^a,-yq^b;q^m)_\infty(txq^a,t^{-1}xq^b;q^m)_\infty}.
\end{equation*}
Using identities \eqref{jacobi}--\eqref{lemfrphif}, we have:
\begin{equation}\label{eqmmm1}
  \sum_{\mu\in\cD}t^{\ell_{a,m}(\mu)-\ell_{m-a,m}(\mu)}q^{|\mu|}=
  \frac{(-q;q)_\infty}{(-q^a,-q^{m-a},q^m;q^m)_\infty}
  \sum_{n\in\bZ}q^{\frac{mn(n-1)}{2}+n a}t^{n},
\end{equation}
\begin{equation}\label{eqmmm2}
  \sum_{\mu\in\cP}t^{\ell_{a,m}(\mu)-\ell_{m-a,m}(\mu)}q^{|\mu|}=
  \frac{(q^a,q^{m-a};q^m)_\infty}{(q;q)_\infty(q^m;q^m)_\infty^2}
  \sum_{n\in\bZ}\frac{(-1)^{n}q^{m\frac{n(n+1)}{2}}}{1-q^{mn+a}t},
\end{equation}
and
\begin{align}\label{eqmmm3}
  \sum_{(\lambda, \mu)\in\cP\times \cD}
  t^{\ell_{a,m}(\lambda,\mu)-\ell_{m-a,m}(\lambda,\mu)}q^{|\lambda|+|\mu|}
  &=\frac{(q^2;q^2)_\infty(q^{2m};q^{2m})_\infty^2(q^a,q^{m-a};q^m)_\infty}
    {(q;q)_\infty^2(q^m;q^m)_\infty^4(-q^a,-q^{m-a};q^m)_\infty}\nonumber\\
  &\quad\times\left(1+2\sum_{n\ge 1}\frac{q^{an}t^{n}+q^{(m-a)n}t^{-n}}
    {1+q^{mn}}\right).
\end{align}
Therefore, by collecting the coefficients of $t^n$ for $n \in \mathbb{N}$
(and omitting the other coefficients where $n$ is negative)
in the above equations \eqref{eqmmm1}--\eqref{eqmmm3},
we readily obtain the proof of the lemma.
\end{proof}

\subsection{Asymptotics of the generating functions}
In this subsection, we study the asymptotics of the $q$-series
$G_{xy}(a,m;q)$, with $(x,y)\in S$, $q$ nearing 
the $m$-th
root of unity $\zeta_{m}^h:=e^{2\pi\ri h/m}$, where $m\in\bN$ and $h\in\bN_0$.
In order to achieve this, we first need the following proposition
which are special cases of results by Katsurada
\cite[Theorem 5, Corollary 1.1, Corollary 1.3]{KMAs}.
\begin{proposition}\label{propm2}
  The following asymptotic formulas holds for $z\to 0$ in every
  fixed Stolz angle $|\arg (z)|\le \Delta$ $(0<\Delta<\pi/2)$.
\begin{enumerate}
  \item For any $h\in\bN_0$ and $m\in\bN$ such that $\gcd(h,m)=1$, we have
    $$(e^{-z}\zeta_{m}^h;e^{-z}\zeta_{m}^h)_\infty=
    A_0(h,m)z^{-1/2}\exp\left(-\frac{\pi^2}{6m^2z}\right)(1+O(z)),$$
where $A_0(h,m)$ is a constant. In particular, $A_0(1,1)=\sqrt{2\pi}$.
  \item For any $\alpha>0$ we have
    $$(-e^{-\alpha z};e^{-z})_\infty=
    2^{1/2-\alpha}\exp\left(\frac{\pi^2}{12z}\right)(1+O(z)).$$
  \item For any $\alpha, \mu\in(0,1)$,
\begin{equation*}
  (e^{-\alpha z}, e^{-(1-\alpha) z};e^{-z})_{\infty}=
  2\sin(\pi\alpha)\exp\left(-\frac{\pi^2}{3z}\right)(1+O(z)),
\end{equation*}
and
\begin{equation*}
  (e^{2\pi\ri \mu}e^{-\alpha z},e^{2\pi\ri (1-\mu)}e^{-(1-\alpha) z};e^{-z})_{\infty}
  =e^{-2\pi\ri (1/2-\alpha)(1/2-\mu)}\exp\left(-\frac{2\pi^2B_2(\mu)}{z}\right)
  (1+O(z)),
\end{equation*}
where $B_2(\mu)=\mu^2-\mu+1/6$.\footnote{Here we used the fact that
  $$\frac{1}{2}\left(\zeta_\mu(2)+\zeta_{1-\mu}(2)\right)=
  \sum_{n\ge 1}\frac{e^{2\pi\ri\mu n}+e^{2\pi\ri(1-\mu) n}}{2n^2}=
  \sum_{n\ge 1}\frac{\cos(2\pi n \mu)}{n^2}=\pi^2B_2(\mu),$$
  which can be found in \cite[Eq.~25.12.8]{NIST:DLMF}, where
  $\zeta_\mu(s):=\sum_{n\ge 1}\frac{e^{2\pi\ri\mu n}}{n^s}$
  is the periodic zeta function.}
\end{enumerate}
\end{proposition}

The asymptotics of the infinite products in Lemma \ref{lemmm1} will follow from
Proposition~\ref{propm2}. It remains to establish the asymptotics for the
infinite sums in Lemma~\ref{lemmm1}. In the following, let ${\bf 1}_{event}$
denote the indicator function, let $h\in\bN_0$ and let $q=e^{-z/m}\zeta_m^h$.
Let $\psi(u)=\Gamma'(u)/\Gamma(u)$ be the digamma function given by
$$\psi(u)=-\gamma+\sum_{n\ge 1}\left(\frac{1}{n}-\frac{1}{n+u}\right).$$
We prove the following lemmas.
\begin{lemma}\label{lem33}As $z\to 0$ in every fixed Stolz angle
  $|\arg (z)|\le \Delta$ $(0<\Delta<\pi/2)$, we have
\begin{equation}\label{eqdds}
  \sum_{n\ge 1}q^{m\frac{n(n-1)}{2}+n a}=
  \frac{{\bf 1}_{m\mid ah}}{2}\left(\frac{2\pi}{z}\right)^{1/2}+O(1),
\end{equation}
\begin{equation}\label{eqpds}
  \sum_{n\ge 0}(-1)^{n}\frac{q^{m\frac{n(n+1)}{2}+mn+a}}
  {1-q^{mn+a}}=\frac{{\bf 1}_{m\mid ah}}{2z}
  \left(\psi\left(\frac{m+a}{2m}\right)-
    \psi\left(\frac{a}{2m}\right)\right)+O(|z|^{-1/2}),
\end{equation}
and
\begin{equation}\label{eqopds}
  \sum_{n\ge 1}\frac{q^{an}}{1+q^{mn}}=
  \frac{{\bf 1}_{m\mid ah}}{2z}\left(\psi\left(\frac{m+a}{2m}\right)-
    \psi\left(\frac{a}{2m}\right)\right)+O(1).
\end{equation}
\end{lemma}
\begin{proof}Let $z=x+\ri y$ with $x>0$ and $y\in\bR$, then we have
  $|y|\le x\cdot\tan (\Delta)$.  The proofs of \eqref{eqdds} and
  \eqref{eqopds} use similar arguments. In fact, note that
$$S(u):=\sum_{1\le n\le u}\zeta_{m}^{ahn}=u\cdot {\bf 1}_{m\mid ah}+O(1),$$
uniformly for all $u>0$. Hence, Abel's summation formula implies
\begin{align*}
  \sum_{n\ge 1}e^{-z\frac{n(n-1)}{2}-\frac{an}{m}z}\zeta_{m}^{ahn}
  &=S(u)e^{-\frac{u(u-1)}{2}z-\frac{a}{m}uz}\big|_{u=0}^{+\infty}-\int_{0}^{\infty}S(u)\left(e^{-\frac{u(u-1)}{2}z-\frac{a}{m}uz}\right)'\rd u\\
  &={\bf 1}_{m\mid ah}\cdot \int_{0}^{\infty}u
    \left(e^{-\frac{u(u-1)}{2}z-\frac{a}{m}uz}\right)'\rd u+
    O\left(\int_{0}^{\infty}
    \left|\left(e^{-\frac{u(u-1)}{2}z-\frac{a}{m}uz}\right)'\right|\rd u\right)\\
  &={\bf 1}_{m\mid ah}\cdot
    \int_{0}^{\infty}e^{-\frac{u(u-1)}{2}z-\frac{a}{m}uz}\rd u+
    O\left(|z|\int_{0}^{\infty}\left|u-\frac{1}{2}+
    \frac{a}{m}\right|e^{-\frac{u(u-1)}{2}x-\frac{a}{m}ux}\rd x\right)\\
  &={\bf 1}_{m\mid ah}\cdot \int_{\frac{a}{m}-\frac{1}{2}}^{\infty}
    e^{-\frac{u^2z}{2}+\frac{(a/m-1/2)^2}{2}z}\rd u+
    O\left(|z|\int_{\frac{a}{m}-\frac{1}{2}}^{\infty}|u|
    e^{-\frac{u^2x}{2}+\frac{(a/m-1/2)^2}{2}x}\rd u\right)\\
&=\frac{{\bf 1}_{m\mid ah}}{2}\left(\frac{2\pi}{z}\right)^{1/2}+O(1).
\end{align*}
Abel's summation formula further implies
\begin{align*}
  \sum_{n\ge 1}\frac{e^{-azn/m}}{1+e^{-nz}}e^{2\pi\ri ahn/m}
  &=S(u)\frac{e^{-azu/m}}{1+e^{-uz}}\bigg|_{u=0}^{+\infty}-
    \int_{0}^{\infty}S(u)\left(\frac{e^{-azu/m}}{1+e^{-uz}}\right)'\rd u\\
  &={\bf 1}_{m\mid ah}\cdot \int_{0}^{\infty}u
    \left(\frac{e^{-azu/m}}{1+e^{-uz}}\right)'\rd u+
    O\left(\int_{0}^{\infty}\left|\left(\frac{e^{-azu/m}}
    {1+e^{-uz}}\right)'\right|\rd u\right)\\
  &={\bf 1}_{m\mid ah}\cdot \int_{0}^{\infty}\frac{e^{-azu/m}}
    {1+e^{-uz}}\rd u+O\left(|z|\int_{0}^{\infty}
    \frac{e^{aux/m}+e^{-(1-a/m)ux}}{|e^{auz/m}+e^{-(1-a/m)uz}|^2}\rd x\right)\\
  &=\frac{{\bf 1}_{m\mid ah}}{z}\int_{0}^{\infty}
    \sum_{k\ge 0}(-1)^{k}e^{-(k+a/m)u}\rd u+
    O\left(1+\int_{1/|z|}^{\infty}\frac{|z|e^{-aux/m}\rd u}{(1-e^{-ux})^2}\right)\\
&=\frac{{\bf 1}_{m\mid ah}}{z}\sum_{k\ge 0}\frac{(-1)^{k}}{k+a/m}+O(1).
\end{align*}
By using \cite[Eq.~5.7.7]{NIST:DLMF}, we have
\begin{equation}\label{eqmmpsi1}
  \sum_{k\ge 0}\frac{(-1)^k}{k+a/m}=
  \frac{1}{2}\left(\psi\left(\frac{m+a}{2m}\right)-
    \psi\left(\frac{a}{2m}\right)\right),
\end{equation}
which completes the proof of \eqref{eqopds}.

It remains to give the proof of \eqref{eqpds}. Note that
\begin{equation*}
  \sum_{n\ge 0}(-1)^{n}\frac{q^{m\frac{n(n+1)}{2}+mn+a}}
  {1-q^{mn+a}}=\sum_{k\ge 0}\left(\frac{(1-q^m)q^{mk(2k+1)-a}}
    {(q^{-(2mk+a)}-1)(q^{-(2km+m+a)}-1)}+
    \frac{(1-q^{m(2k+1)})q^{mk(2k+1)}}{q^{-(2mk+a)}-1}\right).
\end{equation*}
We have
\begin{align*}
  \sum_{k\ge 0}\frac{(1-q^{m(2k+1)})q^{mk(2k+1)}}{q^{-(2mk+a)}-1}
  &\ll \sum_{k\ge 0}\frac{|1-e^{-(2k+1)z}|e^{-k^2x}}{e^{(2k+a/m)x}-1}\\
  &\ll \sum_{0\le k\le 1/|z|}\frac{(2k+1)|z|e^{-k^2x}}{(2k+a/m)x}+
    \sum_{k\ge 1/|z|}e^{-k^2x}\\
  &\ll \sum_{k\ge 0}e^{-k^2x}\ll x^{-1/2}\ll |z|^{-1/2},
\end{align*}
and
\begin{align*}
  \sum_{k\ge 0}\frac{(1-q^m)q^{mk(2k+1)-a}}{(q^{-(2mk+a)}-1)(q^{-(2km+m+a)}-1)}=
  &\sum_{0\le k\le |z|^{-1/2}}\frac{(1-e^{-z})e^{-z(k(2k+1)-a/m)}\zeta_m^{-ah}}
    {(e^{(2k+a/m)z}\zeta_{m}^{ah}-1)(e^{(2k+1+a/m)z}\zeta_{m}^{ah}-1)}\\
&+O\left(\sum_{k\ge |z|^{-1/2}}\frac{|z|e^{-k^2x}}{(2k+a/m)(2k+1+a/m)x^2}\right).
\end{align*}
For the $O$-term, we estimate that
$$\sum_{k\ge |z|^{-1/2}}\frac{|z|e^{-k^2x}}
{(2k+a/m)(2k+1+a/m)x^2}\ll \frac{1}{|z|}
\sum_{k\ge |z|^{-1/2}}\frac{1}{k^2}\ll |z|^{-1+1/2}\ll |z|^{-1/2}.$$
We now estimate the  main term. Since $0\le k\le |z|^{-1/2}$,
for $m\nmid ah$, we have
\begin{equation*}
  \sum_{0\le k\le |z|^{-1/2}}\frac{(1-e^{-z})e^{-z(k(2k+1)-a/m)}\zeta_m^{-ah}}
  {(e^{(2k+a/m)z}\zeta_{m}^{ah}-1)(e^{(2k+1+a/m)z}\zeta_{m}^{ah}-1)}\ll
  \sum_{0\le k\le |z|^{-1/2}}\frac{|z|}{\left||\zeta_{m}^{ah}-1|+
      O(z^{1/2})\right|^2}\ll |z|^{1/2}.
\end{equation*}
For $m\mid ah$, we have
\begin{align*}
  &\sum_{0\le k\le |z|^{-1/2}}\frac{(1-e^{-z})e^{-z(k(2k+1)-a/m)}\zeta_m^{-ah}}
    {(e^{(2k+a/m)z}\zeta_{m}^{ah}-1)(e^{(2k+1+a/m)z}\zeta_{m}^{ah}-1)}\\
  &=\sum_{0\le k\le |z|^{-1/2}}\frac{z}{(2k+a/m)(2k+1+a/m)z^2}
    (1+O((k+1)^2z))(1+O((k+1)z))\\
  &=\frac{1}{z}\sum_{0\le k\le |z|^{-1/2}}\frac{1}{(2k+a/m)(2k+1+a/m)}+
    O\left(\sum_{0\le k\le |z|^{-1/2}}1\right)=\frac{1}{z}\sum_{k\ge 0}
    \frac{(-1)^k}{k+a/m}+O(|z|^{-1/2}).
\end{align*}
Combining the above estimates we get
$$\sum_{n\ge 0}(-1)^{n}\frac{q^{m\frac{n(n+1)}{2}+mn+a}}{1-q^{mn+a}}=
\frac{{\bf 1}_{m\mid ah}}{2z}\left(\psi\left(\frac{m+a}{2m}\right)-
  \psi\left(\frac{a}{2m}\right)\right)+O(|z|^{-1/2}),$$
which completes the proof of the lemma.
\end{proof}

\begin{lemma}\label{lemm34} We have
\begin{equation*}
  G_{01}\!\left(a,m;e^{-z/m}\right)=
  \frac{1+O(|z|^{1/2})}{2}
  \frac{1}{\sqrt{2}}\exp\left(\frac{\pi^2m}{12z}\right),
\end{equation*}
\begin{equation*}
  G_{10}\!\left(a,m;e^{-z/m}\right)=(1+O(|z|^{1/2}))
  \frac{\psi\left(\frac{m+a}{2m}\right)-
    \psi\left(\frac{a}{2m}\right)}{2\pi\csc(a\pi/m)}
  \left(\frac{z}{2\pi m}\right)^{1/2}\exp\left(\frac{\pi^2m}{6z}\right),
\end{equation*}
and
\begin{equation*}
  G_{11}\!\left(a,m;e^{-z/m}\right)=(1+O(|z|^{1/2}))
  \frac{\psi\left(\frac{m+a}{2m}\right)-
    \psi\left(\frac{a}{2m}\right)}{2\pi\csc(a\pi/m)}
  \left(\frac{z}{4\pi m}\right)^{1/2}\exp\left(\frac{\pi^2m}{4z}\right)
\end{equation*}
when $z\to 0$ in every fixed Stolz angle $|\arg (z)|\le \Delta$
$(0<\Delta<\pi/2)$.
\end{lemma}
\begin{proof}
  This proof of the lemma can be straightforwardly done using
  Proposition~\ref{propm2} and Lemma~\ref{lem33}; we omit the details.
\end{proof}
\begin{lemma}\label{proop}For $(x,y)\in S$, we have
\begin{equation*}
  G_{xy}\!\left(a,m; e^{-z/m}\zeta_{m}^h\right)=
  \left({\bf 1}_{m\mid h}+O(|z|^{1/2})\right)G_{xy}\!\left(a,m;e^{-z/m}\right),
\end{equation*}
when $z\rrw 0$ in every fixed Stolz angle $|\arg (z)|\le \Delta$
$(0<\Delta<\pi/2)$.
\end{lemma}
\begin{proof}
  This proof of the lemma can be straightforwardly done using
  Proposition~\ref{propm2}, Lemmas~\ref{lem33} and \ref{lemm34}.
  We sketch some of the details. For $q=e^{-z/m}\zeta_{m}^h$ with
  $m\nmid h$, we have
\begin{align*}
  G_{01}\!\left(a,m;e^{-z/m}\zeta_{m}^h\right)
  &=\frac{(e^{-2z/m}\zeta_m^{2h};e^{-2z/m}\zeta_m^{2h})_\infty}
    {(e^{-z/m}\zeta_m^{h};e^{-z/m}\zeta_m^{h})_\infty}
    \frac{\sum_{n\ge 1}q^{m\frac{n(n-1)}{2}+an}}
    {(e^{-az/m}\zeta_{2m}^{m+2ah},e^{-(1-a/m)z}\zeta_{2m}^{-m-2ah},
    e^{-z};e^{-z})_\infty}\\
  &\ll\left(|z|^{1/2}+{\bf 1}_{m\mid ah}\right)\exp\left(\frac{\gcd(m,h)^2}
    {6mz/\pi^2}+\frac{\pi^2}{6z}+
    \frac{B_2\left(\left\{\frac{2ah+m}{2m}\right\}\right)}{z/(2\pi^2)}-
    \frac{\gcd(m,2h)^2}{12mz/\pi^2}\right)\\
  &\ll\left(|z|^{1/2}+{\bf 1}_{m\mid ah}\right)
    G_{01}\!\left(a,m;e^{-z/m}\right)\exp\left(-\frac{2\pi^2}{z}
    A_{01}(a,m; h)\right),
\end{align*}
where
$$A_{01}(a,m; h)=\frac{m}{24}+\frac{\gcd(m,2h)^2}{24m}-
\frac{\gcd(m,h)^2}{12m}-\frac{1}{12}-
B_2\left(\left\{\frac{2ah+m}{2m}\right\}\right).$$
Similarly, we have
\begin{align*}
  G_{10}\!\left(a,m;e^{-z/m}\zeta_{m}^h\right)
  &=\frac{(\zeta_{m}^{ah}e^{-az/m},\zeta_{m}^{-ah}e^{-(1-a/m)z};e^{-z})_\infty}
    {(e^{-z};e^{-z})_\infty^2( e^{-z/m}\zeta_m^{h};e^{-z/m}\zeta_m^h)_\infty}
    \sum_{k\ge 0}(-1)^{k}\frac{q^{m\frac{k(k+1)}{2}+mk+a}}{1-q^{mk+a}}\\
  &\ll \left(1+|z|^{-1/2}{\bf 1}_{m\mid ah}\right)
    |z|\exp\left(\frac{\pi^2}{3z}+\frac{\pi^2\gcd(m,h)^2}{6mz}-
    \frac{2\pi^2}{z}B_2\left(\left\{\frac{a h}{m}\right\}\right)\right)\\
  &\ll \left(|z|^{1/2}+{\bf 1}_{m\mid ah}\right)
    G_{10}\!\left(a,m;e^{-z/m}\right)\exp\left(-\frac{2\pi^2}{z}
    A_{10}(a,m; h)\right),
\end{align*}
where
$$A_{10}(a,m; h)=\frac{m}{12}+B_2\left(\left\{\frac{a h}{m}\right\}\right)-
\frac{1}{6}-\frac{\gcd(m,h)^2}{12m};$$
and
\begin{align*}
  G_{11}\!\left(a,m;e^{-z/m}\zeta_{m}^h\right)
  &=2\frac{(e^{-2z/m}\zeta_{m}^{2h};e^{-2z/m}\zeta_{m}^{2h})_\infty
    (e^{-2z};e^{-2z})_\infty^2}
    {(e^{-z/m}\zeta_{m}^{h};e^{-z/m}\zeta_{m}^{h})_\infty^2
    (e^{-z};e^{-z})_\infty^4}\\
  &\quad\times \frac{(e^{-az/m}\zeta_{m}^{ah},e^{-(1-a/m)z}
    \zeta_{m}^{-ah};e^{-z})_\infty}
    {(e^{-az/m}\zeta_{2m}^{m+2ah},e^{-(1-a/m)z}\zeta_{2m}^{-m-2ah};
    e^{-z})_\infty}\sum_{n\ge 1}\frac{q^{an}}{1+q^{mn}}\\
  &\ll \left(|z|+{\bf 1}_{m\mid ah}\right) |z|^{1/2}
    \exp\left(\frac{\pi^2\gcd(m,h)^2}{3mz}-
    \frac{\pi^2\gcd(m,2h)^2}{12mz}+\frac{\pi^2}{2z}\right)\\
  &\quad\times\exp\left(-\frac{2\pi^2}{z}
    \left(B_2\left(\left\{\frac{a h}{m}\right\}\right)-
    B_2\left(\left\{\frac{2ah+m}{2m}\right\}\right)\right)\right)\\
  &\ll \left(|z|+{\bf 1}_{m\mid ah}\right)
    G_{11}\!\left(a,m;e^{-z/m}\right)\exp\left(-\frac{2\pi^2}{z}
    A_{11}(a,m;h)\right),
\end{align*}
where
$$A_{11}(a,m;h)=\frac{m}{8}+\frac{\gcd(m,2h)^2}{24m}+
B_2\left(\left\{\frac{a h}{m}\right\}\right)-
\frac{1}{4}-\frac{\gcd(m,h)^2}{6m}-
B_2\left(\left\{\frac{2ah+m}{2m}\right\}\right).$$
It is clear that $A_{11}(a,m;h)=A_{01}(a,m;h)+A_{10}(a,m;h)$. Using
elementary arguments one can show that $A_{xy}(a,m;h)$ with $(x,y)\in S$,
is non-negative for all integers $m\nmid h$ and $m\ge 3$.  Moreover, for
$m\mid ah$ but $m\nmid h$, we have
$$A_{01}(a,m; h)=\frac{m}{24}+\frac{\gcd(m,2h)^2}{24m}-
\frac{\gcd(m,h)^2}{12m}\ge \frac{m}{24}-\frac{\gcd(m,h)^2}{24m}\ge
\left(1-\frac{1}{4}\right)\frac{m}{24}>0,$$
and
$$A_{10}(a,m; h)=\frac{m}{12}-\frac{\gcd(m,h)^2}{12m}\ge
\left(1-\frac{1}{4}\right)\frac{m}{12}>0,$$
because of $\gcd(m,2h)\ge \gcd(m,h)$ and $\gcd(m,h)\le m/2$. The combination
of all the above estimates proves the lemma.
\end{proof}

\subsection{Proof of Theorems \ref{th2} and \ref{propit}}\label{sec31}
In this subsection, we will provide the proof of Theorem~\ref{th2}.
To accomplish this, we will first establish Theorem~\ref{propit},
the residue class analogue of Ingham's Tauberian theorem.
A special case of Ingham's Tauberian theorem, as presented in
\cite[p.1082, Eqs.(21),(22)]{MR5522}, is stated as the following theorem.

\begin{theorem}\label{thitt}Let $A(u)$ be an increasing function on
  $[0, \infty)$ with $A(0)=0$ be defined as the following Laplace transform:
\begin{equation}\label{eqm400}
f(s)=\int_{0}^{\infty}e^{-us}\rd A(u).
\end{equation}
If $f(s)\sim f_0(s)$ when $s\rrw 0$ in every fixed Stolz angle
$|\arg(s)|\le \Delta$ $(0<\Delta<\pi/2)$, where
$$f_0(s)=C(M/s)^{v\beta-\frac{1}{2}}e^{\beta^{-1}(M/s)^{\beta}}\quad
(\beta, M, C>0, v\in\bR).$$
Then with $\alpha=\beta/(1+\beta)$ we have:
$$A(\omega)\sim \left(\frac{1-\alpha}{2\pi}\right)^{\frac{1}{2}}
C(M\omega)^{v\alpha-\frac{1}{2}}e^{\alpha^{-1}(M\omega)^{\alpha}},
\text{ as }\omega\rrw+\infty.$$
\end{theorem}

Let $F(q)=\sum_{n\ge 0}c_nq^n$, be a power series whose coefficients $c_n$
are non-negative and increasing, and converges absolutely for $|q|<1$.
Notice that \eqref{eqm400} involves a Riemann--Stieltjes integral.
Therefore, if we define
$$A(u)=c_{\lfloor u\rfloor}-c_{0},$$
for any $u\ge 0$, then $A(u)$ is non-negative and increasing with $A(0)=0$.
Integration by parts for a Riemann--Stieltjes integral yields
\begin{equation}\label{eqm40}
  \int_{0}^{\infty}e^{-us}\rd A(u)=s\int_{0}^{\infty}A(u)e^{-us}\rd u=
  \sum_{k\ge 0}A(k)\int_{k}^{k+1}s e^{-us}\rd u=(1-e^{-s})F(e^{-s})-c_0.
\end{equation}
As a conclusion, if $F(e^{-s})$ meets the conditions in Theorem~\ref{thitt},
then we can use that theorem to derive an asymptotic formula for $c_n$.

However, in our case it is not easy to prove that $p_n(a,m-a,m;x,y)$
is increasing for $n\ge 0$. We only have
\begin{equation}\label{eqmmot}
p_{n+m}(a,m-a,m;x,y)\ge p_{n}(a,m-a,m;x,y),
\end{equation}
for all $n\ge 0$, due to Lemma~\ref{promt}. Luckily, we can prove
Theorem~\ref{propit}, a residue class analogue of Ingham's Tauberian
theorem, which works in our case.

\begin{proof}[Proof of Theorem~\ref{propit}]Due to condition~\eqref{c1}
  of Theorem~\ref{propit}, $c_{mn+j}$ is increasing in $n$,
  for any $0\le j<m$. Moreover, we have
  \begin{equation*}
    \sum_{n\ge 0}c_{mn+j}q^{n}=\frac{1}{m}\sum_{0\le h<m}e^{-2\pi\ri hj}q^{-j/m}
    f\!\left(q^{1/m}e^{2\pi\ri h/m}\right),
\end{equation*}
by using the orthogonality of the $m$-th roots of unity.
Using condition \eqref{c2} of Theorem \ref{propit} and the above,
with $A_{j}(u)=c_{m\lfloor u\rfloor+j}-c_j$, we have
$$c_j+\int_{0}^{\infty}e^{-uz}\rd A_{j}(u)=
(1-e^{-z})\sum_{n\ge 0}c_{mn+j}e^{-nz}\sim \frac{\alpha }
{m^{1+\gamma}} z^{1+\gamma}e^{m^\rho\beta z^{-\rho}/\rho},$$
by using \eqref{eqm40}, when $z\rrw 0$ in every fixed Stolz angle
$|\arg (z)|\le \Delta$, $(0<\Delta<\pi/2)$. Hence
Theorem~\ref{thitt} implies
$$c_{mn+j}\sim \frac{\alpha \beta^{\frac{1+2\gamma}{2(1+\rho)}}}
{\sqrt{2\pi(1+\rho)}}(mn)^{-\frac{1+2\gamma}{2(1+\rho)}-
  \frac{1}{2}}e^{(1+1/\rho)\beta^{1/(1+\rho)} (mn)^{\rho/(1+\rho)}},$$
as $n\rrw+\infty$. Therefore, by using the simple estimates
$$(n+r)^{\kappa_1}\sim n^{\kappa_1}\quad\text{and}\quad
(n+r)^{\kappa_2}=n^{\kappa_2}+O(n^{\kappa_2-1})=n^{\kappa_2}+o(1),$$
as $n\rrw \infty$, for any given $\kappa_1\in\bR$ and $\kappa_2<1$,
we can replace $mn$ by $nm+j$ in above asymptotic formula of $c_f(mn+j)$.
This completes the proof of Theorem~\ref{propit}.
\end{proof}

We now give the proof of Theorem~\ref{th2}. From inequality~\eqref{eqmmot},
Lemmas~\ref{lemm34} and \ref{proop}, we see that the generating function
$G_{xy}(a,m;q)$ with $(x,y)\in S$ meets the conditions of
Theorem~\ref{propit}. Thus we have
$$p_{n}(a,m-a,m; x,y)\sim c_{a,m}(x)\widehat{p}_n(x,y),$$
as $n\rrw +\infty$, with
$$\left(c_{a,m}(0), c_{a,m}(1)\right)=
\left(\frac{1}{2},~\frac{\psi\left(1/2+a/2m\right)-
    \psi\left({a}/{2m}\right)}{2\pi\csc(a\pi/m)}\right).$$
Here
$$\widehat{p}_n(0,1)=\frac{e^{\pi\sqrt{n/3}}}
{4\sqrt[4]{3}n^{3/4}},\;\; \widehat{p}_n(1,0)=
\frac{e^{2\pi\sqrt{n/6}}}{4\sqrt{3}n},\;\; \text{and}\;\; \widehat{p}_n(1,1)=
\frac{1}{8n}e^{\pi\sqrt{n}}.$$
Recall the unrestricted partition function $p(n)$
(see \cite[Eq.~(1.41)]{MR1575586}), the distinct partition function
$q(n)$ (see \cite[pages 109--110]{MR1575586}), and the overpartition
function $\bar{p}(n)$ (see \cite[Eq.~(1)]{MR5522}), which have the
following asymptotic formulas:
$$q(n)\sim \frac{e^{\pi\sqrt{n/3}}}
{4\sqrt[4]{3}n^{3/4}}, \;\; p(n)\sim \frac{e^{2\pi\sqrt{n/6}}}
{4\sqrt{3}n}, \;\; \text{and} \;\; \bar{p}(n)\sim \frac{1}{8n}e^{\pi\sqrt{n}},$$
as $n\to\infty$. This completes the proof of Theorem~\ref{th2}.

\section{Remarks on residue class biases in distinct partitions}\label{sec4}
Theorem~\ref{mth1} provides results for biases in residue classes in
partitions and overpartitions. For distinct partitions, we only have a
partial result, given in Theorem~\ref{mth2}. Beyond that partial result,
we provide a conjecture, proposed in the end of this section.

In the following, we write $d_n(a,b;m):=p_n(a,b,m; 0,1)$ for brevity. Then
$$\sum_{n\ge 0}d_n(a,b;m)q^n=
\sum_{\substack{\lambda\in \cD\\ \ell_{a,m}(\lambda)> \ell_{b,m}(\lambda)}}
q^{|\lambda|}.$$
If $b=m-a$, then Lemma~\ref{lemmm1} implies
\begin{equation}\label{eqpro31}
  \sum_{n\ge 0}d_n(a,m-a;m)q^{n}=\frac{(-q;q)_\infty}
  {(-q^a,-q^{m-a},q^m;q^m)_\infty}\sum_{n\ge 1}q^{m\binom{n}{2}+n a}.
\end{equation}
From Equation \eqref{eqpro31} we have the following inequalities
for $d_{n}(1,2;3)$ and $d_{n}(2,1;3)$.
\begin{proposition}We have
$$d_{3n+2}(1,2;3)\le d_{3n+2}(2,1;3),$$
$$d_{3n}(1,2;3)\ge d_{3n}(2,1;3), \quad\text{and}\quad
d_{3n+1}(1,2;3)\ge d_{3n+1}(2,1;3),$$
for all $n\ge 0$.
\end{proposition}
\begin{proof}
  The proof is straightforward and uses a dissection strategy.
  First, we observe
\begin{equation*}
  \sum_{n\ge 0}(d_{n}(1,2;3)-d_{n}(2,1;3))q^n=
  \frac{(-q^3;q^3)_\infty}{(q^3;q^3)_\infty}
  \sum_{n\ge 1}q^{\frac{n(3n-1)}{2}}(1-q^{n}).
\end{equation*}
Dissection of the last sum gives
\begin{align*}
  \sum_{n\ge 0}q^{\frac{n(3n-1)}{2}}(1-q^{n})
  &=\sum_{0\le j\le 2}\sum_{n\ge 0}q^{\frac{3(3n+j)-3n+j)}{2}}(1-q^{3n+j})\\
  &=\sum_{n\ge 0}q^{\frac{3n(9n-1)}{2}}(1-q^{3n})+
    q\sum_{n\ge 0}q^{\frac{9n(3n+5)}{2}}(1-q^{12n+6})-
    q^2\sum_{n\ge 0}q^{\frac{9n(3n+7)}{2}}(1-q^{6n+3}).
\end{align*}
This completes the proof.
\end{proof}
A similar argument can be applied to compare $d_{n}(1,3;4)$ and $d_{n}(3,1;4)$:
\begin{align*}
  \sum_{n\ge 0}\left(d_n(1,3;4)-d_n(3,1;4)\right)q^{n}
  &=\frac{(-q;q)_\infty}{(-q,-q^{3},q^4;q^4)_\infty}
    \sum_{n\ge 1}q^{2n^2-n}(1-q^{n})\\
  &=\frac{1}{(q^2;q^2)_\infty}\sum_{n\ge 1}q^{2n^2-n}(1-q^{2n})
    =\frac{1}{(q^4;q^2)_\infty}\sum_{n\ge 1}q^{2n^2}\sum_{0\le j<n}q^{2j-n}.
\end{align*}
Therefore, $d_n(1,3;4)\ge d_n(3,1;4)$.
This is also a special case of Theorem~\ref{mth2}. 

We turn to the general case.
For any $1\le a<b\le m$, by Equation~\eqref{eqmgp} we have
\begin{equation*}
  \sum_{n\ge 0}d_{n}(a,b;m)q^n=
  \frac{(-q;q)_\infty}{(-q^a,-q^b;q^m)_\infty}
  \sum_{\substack{j> k\\ j,k\ge 0}}\frac{q^{\frac{1}{2}m(j(j-1)+k(k-1))+aj+bk}}
  {(q^m;q^m)_{j}(q^m;q^m)_{k}}.
\end{equation*}
Based on this generating function and extensive computer experiments,
we propose the following conjecture.
\begin{conjecture}\label{conj1}For any positive integers $a,b,m$ such
  that $1\le a<b\le m$, $m\ge 3$ and $(a,b,m)\neq (1,2,3)$, there exists
  a constant $N_{a,b,m}>0$ such that
$$d_{n}(a,b; m)\ge d_{n}(b,a; m),\quad\text{for all}\quad n\ge N_{a,b, m}.$$
Moreover,
$$N_{a,m-a,m}=0,\quad\text{for all}\quad 1\le a<m/2,$$
except
$$N_{2,3,5}=45, N_{2,4,6}=5, N_{3,4,7}=8, N_{4,5,9}=9.$$
\end{conjecture}
\begin{remark}The constants $N_{a,b, m}$ appear to be all
  relatively small. In particular, there exists an $m_0>9$
  such that $N_{a,b, m}$ equals zero for all $m\ge m_0$.
\end{remark}
\begin{remark}
  We note that we confirmed Conjecture~\ref{conj1} for a class
  of special cases in Theorem~\ref{mth2}.
\end{remark}

\section*{Statements and Declarations}
The authors declare no conflicts of interest.
Both authors contributed to the preparation, writing, reviewing and
editing of the manuscript.
Data sharing not applicable to this article.


\end{document}